\RequirePackage{fix-cm}
\documentclass[smallextended,envcountsect]{svjour3}       
\smartqed  

\usepackage{amsmath,amsxtra, amssymb, latexsym, amscd,cite,comment}
\usepackage{graphicx,color}
\usepackage[active]{srcltx}
\usepackage[utf8]{inputenc}
\usepackage[mathscr]{euscript}
\usepackage[english]{babel}
\usepackage{color,enumerate,hyperref}
\begin{document}

\title{Locally Lipschitz vector optimization problems: second-order constraint qualifications, regularity condition and KKT necessary optimality conditions
}
\titlerunning{Locally Lipschitz vector optimization problems}
\author{Yi-Bin Xiao \and Nguyen Van Tuyen  \and Jen-Chih Yao \and Ching-Feng Wen
}
\authorrunning{Y.-B. Xiao, N.V. Tuyen, J.-C. Yao and C.-F. Wen}

\institute{
Yi-Bin Xiao \at School of Mathematical Sciences, University of Electronic Science and Technology of China, Chengdu, P.R. China,\\
\email{xiaoyb9999@hotmail.com}
\smallskip
           \and
Nguyen Van Tuyen \at School of Mathematical Sciences, University of Electronic Science and Technology of China, Chengdu, P.R. China; Department of Mathematics, Hanoi Pedagogical University 2, Xuan Hoa, Phuc Yen, Vinh Phuc, Vietnam\\
              \email{tuyensp2@yahoo.com; nguyenvantuyen83@hpu2.edu.vn}
              \smallskip
           \and
           Jen-Chih Yao \at
           Center for General Education, China Medical University, Taichung, 40402, Taiwan
           \\
           \email{yaojc@mail.cmu.edu.tw}
           \smallskip
              \and
           Ching-Feng Wen \at
           Corresponding author. Center for Fundamental Science; and Research Center for Nonlinear Analysis and Optimization, Kaohsiung Medical University, Kaohsiung, 80708, Taiwan; Department of Medical Research, Kaohsiung Medical University Hospital,
           Kaohsiung, 80708, Taiwan
           \\
           \email{cfwen@kmu.edu.tw}
           }

\date{Received: date / Accepted: date}

\maketitle

\begin{abstract}
{In the present paper, we are concerned with a class of constrained vector optimization problems, where the objective functions and active constraint functions are locally Lipschitz at the referee point. Some second-order constraint qualifications of Zangwill type, Abadie type and Mangasarian~--~Fromovitz type as well  as a  regularity condition of Abadie type are proposed in a nonsmooth setting. The connections between these proposed conditions are established. They are applied to develop second-order Karush--Kuhn--Tucker necessary optimality conditions for local (weak, Geoffrion properly) efficient  solutions to the considered problem. Examples are also given to illustrate the obtained results.}

\end{abstract}

\section{Introduction}
\label{intro}
In this paper, we are interested in second-order optimality conditions for the following constrained vector optimization  problem
\begin{align*}
& \text{min}\, f(x)\label{problem} \tag{VP}
\\
&\text{subject to}\ \ x\in Q_0:=\{x\in X\,:\, g(x)\leqq 0\},
\end{align*}
where $f:=(f_i)$, $i\in I:=\{1, \ldots, p\}$, and $g:=(g_j)$, $j\in J:=\{1, \ldots, m\}$  are vector-valued functions defined on a Banach space $X$.

{As a mainstream in the study of vector optimization problems, optimality condition for vector optimization problems has attracted the attention of many researchers in the field of optimization due to their important applications in many disciplines, such as variational inequalities, equilibrium problems and fixed pointed problems; see, for example, \cite{Lee98,Lu18,mor06,Petrusel18,Wang16,XS,SXC1,SXC2,Qin}.} 

It is well-known that if $f_i$, $g_j$ are differentiable at $\bar x\in Q_0$  and $\bar x$ is a local weak efficient solution of \eqref{problem}, then
there exist Lagrange multipliers $(\lambda, \mu)\in \mathbb{R}^p\times\mathbb{R}^m$ satisfying
\begin{align}
&\sum_{i=1}^p\lambda_i\nabla f_i(\bar x)+\sum_{j=1}^m\mu_j\nabla g_j(\bar x)=0,\label{equa_intro:1}
\\
&\mu=(\mu_1, \ldots, \mu_m)\geqq 0, \mu_jg_j(\bar x)=0,\label{equa_intro:2}
\\
&\lambda=(\lambda_1, \ldots, \lambda_p)\geqq 0, (\lambda, \mu)\neq 0;\label{equa_intro:3}
\end{align}
see \cite[Theorem 7.4]{Jahn04}. Conditions \eqref{equa_intro:1}--\eqref{equa_intro:3} are called the first-order F.-John necessary optimality conditions. If $\lambda$ is nonzero, then these conditions are called the first-order Karush--Kuhn--Tucker $(KKT)$  optimality  conditions.  By Motzkin's theorem of the alternative \cite[p.28]{Mangasarian69}, the existence of $KKT$ multipliers is equivalent to the inconsistency of the following system
\begin{align}
\nabla f_i(\bar x)(v)&<0, \ \ \ i\in I, \label{equa_intro:4}
\\
\nabla g_j(\bar x)(v)&\leqq 0, \ \ \ j\in J(\bar x),\label{equa_intro:5}
\end{align}
with unknown $v\in X$, where $J(\bar x)$ is the  active index set  at $\bar x$.  Conditions \eqref{equa_intro:4}--\eqref{equa_intro:5} are called 
the first-order $KKT$ necessary conditions in primal form.

The first-order $KKT$ optimality conditions are needed to find optimal solutions of constrained optimization problems. In order to obtain these optimality conditions,   constraint qualifications and regularity conditions are indispensable; see, for example, \cite{Andreani11,Tuyen18,Tung-Luu,Luu-Mai,TungLT,Tuyen-Xiao-Son,Movahedian,Soleimani,Gunther}. We recall here that these assumptions are called constraint qualifications $(CQ)$ when they have to be fulfilled by the constraints of the problem, and they are called regularity conditions $(RC)$ when they have to be fulfilled by both the objectives and the constraints of the problem; see \cite{Rizvi12} for more details.

Second-order necessary optimality conditions play an important role in both
the theory and practice of constrained optimization problems. These conditions are used to eliminate nonoptimal KKT points of optimization problems. Moreover, the second-order  optimality condition is  a key tool of numerical analysis in proving convergence and deriving error
estimates for numerical discretizations of optimization problems; see, for example, \cite{Bertsekas99,Izmailov08,Nocedal99}. 

One of the first investigations to obtain second-order optimality conditions of $KKT$-type for smooth vector optimization problems was carried out by Wang \cite{Wang91}. Then, by introducing a new second-order constraint qualification in the sense of Abadie, Aghezzaf et al.  \cite{Aghezzaf99} extended Wang's results to the nonconvex case. Maeda \cite{Maeda04} was the first to propose an Abadie regularity condition and established second-order $KKT$ necessary optimality conditions for $C^{1,1}$ vector optimization
problems. By using the second-order directional derivatives and introducing a new second-order constraint qualification of Zangwill-type, Ivanov \cite{Ivanov15} introduced some optimality conditions for $C^1$ vector optimization problems with inequality constraints.  Very recently, by proposing some types of the second-order Abadie regularity conditions, Huy et al. \cite{Huy162,Huy163} have obtained some second-order $KKT$ necessary optimality conditions for $C^{1,1}$ vector optimization problems in terms of second-order symmetric subdifferentials. For other contributions to second-order $KKT$ optimality conditions for vector optimization, the reader is invited to see the papers \cite{Elena,Ginchev08,Giorgi09,Ivanov152,Ivanov10,Luu17,Kim-Tuyen,Huy-Tuyen} with the references therein.

Our aim is to weaken the hypotheses of the optimality conditions in \cite{Aghezzaf99,Elena,Huy163,Ivanov15,Luu17,Maeda04,Wang91}. To obtain second-order $KKT$ necessary conditions, by using second-order upper generalized directional derivatives and second-order tangent sets, we introduce some second-order constraint qualifications of Zangwill type, Abadie type and Mangasarian-Fromovitz type as well as a  regularity condition of Abadie type.
Our obtained results  improve and generalize the corresponding results in \cite{Aghezzaf99,Elena,Huy163,Ivanov15,Luu17,Maeda04,Wang91}, because the objective functions and the active constraint functions are only locally Lipschitz at the referee point and the required constraint qualifications are also weaker. Moreover, the connections between these proposed conditions are established.

The organization of the paper is as follows. In Section \ref{Preliminaries}, we recall some notations, definitions and preliminary material. Section \ref{Abadie_RC_sect} is devoted to investigate  second-order constraint qualifications and regularity conditions in a nonsmooth setting for vector optimization problems. In Section \ref{Second_order_optim_sect} and Section~\ref{Strong_Second_order_optim_sect}, we establish some second-order necessary optimality conditions of $KKT$-type  for a local  (weak, Geoffrion properly) efficient solution of \eqref{problem}. Section \ref{conclusions_sect}  draws some conclusions.

\section{Preliminaries}
\label{Preliminaries}

In this section, we recall some definitions and introduce basic results, which are useful in our study.

Let $\mathbb{R}^p$ be the $p$-dimensional Euclidean space. For $a, b\in\mathbb{R}^p$, by $a\leqq b$, we mean $a_i\leqq b_i$ for all $i\in I$; by $a\leq b$, we mean $a\leqq b$ and $a\neq b$; and by $a<b$, we mean $a_i<b_i$ for all $i\in I$.

We first recall the definition of local (weak, Geoffrion properly) efficient solutions for the considered problem \eqref{problem}. Note that the concept of properly efficient solution has been introduced at first to eliminate the efficient solutions with unbounded trade-offs. This concept was introduced initially by Kuhn and Tucker \cite{Kuhn50} and was followed thereafter by Geoffrion \cite{Geoffrion68}. Geoffrion's concept enjoys economical interpretations, while Kuhn and Tucker's one is useful for numerical and algorithmic purposes.

\begin{definition}{\rm Let $Q_0$ be the feasible set of \eqref{problem} and $\bar x\in Q_0$. We say that:
\begin{enumerate}[(i)]
\item  $\bar x$ is   {\em an  efficient solution} (resp., {\em a weak efficient solution}) of \eqref{problem}  iff there is no $x\in Q_0$ satisfying  $f(x)\leq f(\bar x)$ (resp., $f(x)<f(\bar x)$).
\item $\bar x$ is a {\em Geoffrion properly efficient solution} of \eqref{problem} iff it is efficient and there exists $M>0$ and  such that, for each $i$,
$$\frac{f_i(x)-f_i(\bar x)}{f_j(\bar x)-f_j(x)}\leqq M,$$
for some $j$ such that $f_j(\bar x)<f_j(x)$ whenever $x\in Q_0$ and $f_i(\bar x)>f_i(x)$.
\item   $\bar x$ is a {\em local efficient solution} (resp., {\em local weak efficient solution, local Geoffrion properly efficient solution}) of \eqref{problem} iff it is an efficient solution (resp., weak efficient solution, Geoffrion properly efficient solution) in $U\cap Q_0$, where $U$ is some neighborhood of $\bar x$.
\end{enumerate}
	}
\end{definition}

Hereafter, we assume that $X$ is a Banach space  equipped with the norm $\|\cdot\|$. Let $\Omega$ be a nonempty subset in $X$. The  {\it closure}, {\it convex hull} and {\it conic hull} of $\Omega$ are denoted by $\mbox{cl}\,\Omega$,
$\mbox{conv}\,\Omega$ and $\mbox{cone}\,\Omega$, respectively.

\begin{definition}{\rm Let   $\bar x\in \Omega$ and $u\in X$.
\begin{enumerate}[(i)]
\item  The {\em tangent cone} to $\Omega$ at $\bar x\in \Omega$ is defined by
$$T(\Omega; \bar x):=\{d\in X\,:\,\exists t_k\downarrow 0, \exists d^k\to d, \bar x+t_kd^k\in \Omega, \ \ \forall k\in \mathbb{N}\}.$$
\item   The {\em second-order tangent set} to $\Omega$ at $\bar x$ with respect to the direction $u$ is defined by
$$T^2(\Omega; \bar x, u):=\left\{v\in X:\exists t_k\downarrow 0, \exists v^k\to v, \bar x+t_ku+\frac12t_k^2v^k\in \Omega,\ \ \forall k\in \mathbb{N}\right\}.$$
\end{enumerate}
	}	
\end{definition}

Clearly, $T(\,\cdot\,; \bar x)$ and $T^2(\,\cdot\,; \bar x, u)$ are isotone, i.e., if $\Omega^1\subset \Omega^2$, then
\begin{align*}
T(\Omega^1; \bar x)&\subset T(\Omega^2; \bar x),
\\
T^2(\Omega^1; \bar x, u)&\subset  T^2(\Omega^2; \bar x, u).
\end{align*}

It is well-known that $T(\Omega; \bar x)$ is a nonempty closed cone. For each $u\in X$, the set $T^2(\Omega; \bar x, u)$ is closed, but may be empty. However, we see that the set $T^2(\Omega; \bar x, 0)=T(\Omega; \bar x)$ is always nonempty.

Let $F\colon X\to\mathbb{R}$ be a real-valued function defined on $X$ and $\bar x\in X$. The function $F$ is said to be {\em locally Lipschitz} at $\bar x$ iff there exist a neighborhood $U$ of $\bar x$ and $L\geqq 0$ such that
\begin{equation*}
|F(x)-F(y)|\leqq L\|x-y\|,\ \ \ \forall x, y\in U.
\end{equation*}
\begin{definition}{\rm Assume that $F\colon X\to\mathbb{R}$ is locally Lipschitz at $\bar x\in X$. Then:
\begin{enumerate}[(i)]
\item  (See \cite{Clarke83}) The {\em Clarke's generalized derivative} of $F$ at $\bar x$ is defined by
\begin{equation*}
F^{\circ} (\bar x, u):= \limsup\limits_{\mathop {x \to \bar x}\limits_{t  \downarrow 0} } \dfrac{F(x+tu)-F(x)}{t}, \ \ \ u\in X.
\end{equation*}
\item   (See \cite{Pales94}) The {\em second-order upper generalized directional derivative} of $F$ at $\bar x$ is defined by
\begin{equation*}
F^{\circ\circ} (\bar x, u):= \limsup\limits_{\mathop {t  \downarrow 0} }\dfrac{F(\bar x+tu)-F(\bar x)-tF^{\circ} (\bar x, u)}{\frac12t^2}, \ \ \ u\in X.
\end{equation*}
\end{enumerate}
	}
\end{definition}
It is easily seen that $F^{\circ} (\bar x, 0)=0$ and $F^{\circ\circ} (\bar x, 0)=0$. Furthermore, the function $u\mapsto F^\circ(\bar x, u)$ is finite, positively homogeneous, and subadditive on $X$; see, for example,  \cite{Clarke83,SX2,XS3}.

The following lemmas will be useful in our study.
\begin{lemma}\label{lemma1} Suppose that $F\colon X\to\mathbb{R}$ is locally Lipschitz at $\bar x\in X$. Let $u\in X$ and let $\{(t_k, u^k)\}$ be a sequence converging to $(0^+, u)$. If
\begin{equation*}
F \left(\bar x+t_ku^k\right)\geqq F(\bar x) \ \ \mbox{for all} \ \ k\in\mathbb{N},
\end{equation*}
then $F^{\circ} (\bar x, u)\geqq 0.$	
\end{lemma}
\begin{proof} Since $F$ is  locally Lipschitz at $\bar x$ and $$\lim\limits_{k\to\infty} (\bar x+t_ku^k)=\lim\limits_{k\to\infty} (\bar x+t_ku)=\bar x,$$
there exist $L\geqq0$ and $k_0\in\mathbb{N}$ such that
$$|F(\bar x+t_ku^k)-F(\bar x+t_ku)|\leqq L t_k\|u^k-u\|\ \ \text{for all}\ \ k\geqq k_0.$$
Thus,
\begin{align*}
0&\leqq F(\bar x+t_ku^k) -F(\bar x)
\\
&= [F(\bar x+t_ku^k)-F(\bar x+t_ku)]+[F(\bar x+t_ku)-F(\bar x)]
\\
&\leqq Lt_k\|u^k-u\| +F(\bar x+t_ku)-F(\bar x)
\end{align*}
for all $k\geqq k_0$. This implies that
\begin{align*}
0&\leqq \lim_{k\to \infty} L \|u^k-u\|+\limsup_{k\to \infty} \dfrac{F(\bar x+t_ku)-F(\bar x)}{t_k}
\\
&\leqq \limsup\limits_{\mathop {x \to \bar x}\limits_{t  \downarrow 0} } \dfrac{F(x+tu)-F(x)}{t}.
\end{align*}
Therefore, $F^{\circ} (\bar x, u)\geqq 0$, as required.
\end{proof}
\begin{lemma}\label{lemma2} Suppose that $F\colon X\to\mathbb{R}$ is locally Lipschitz at $\bar x\in X$. Let $(u, v)$ be a vector in $X\times X$ and let $\{(t_k, v^k)\}$ be a sequence converging to $(0^+, v)$  satisfying
\begin{equation*}
F \left(\bar x+t_ku+\frac12 t^2_kv^k\right)\geqq F(\bar x) \ \ \mbox{for all} \ \ k\in\mathbb{N}.
\end{equation*}
If $F^{\circ} (\bar x, u)= 0$, then $F^{\circ} (\bar x, v)+F^{\circ\circ} (\bar x, u)\geqq 0.$	
\end{lemma}
\begin{proof}  For each $k\in \mathbb{N}$, put $x^k:= \bar x+t_ku+\frac12 t^2_kv^k$ and $y^k:=\bar x+t_ku+\frac12 t^2_kv$. Since $F$ is  locally Lipschitz at $\bar x$ and
$$\lim\limits_{k\to\infty} x^k=\lim\limits_{k\to\infty} y^k=\bar x,$$
there exist $L\geqq 0$ and $k_0\in\mathbb{N}$ such that
\begin{equation*}
|F(x^k)-F\left(y^k\right)|\leqq \frac12t^2_kL\|v^k-v\|\ \ \text{for all}\ \ k\geqq k_0.
\end{equation*}
Thus,
\begin{align*}
0&\leqq F(x^k) - F(\bar x)
\\
&= [F(x^k)-F(y^k)]+[F(y^k) - F(\bar x+t_ku)]
\\
&+  [F(\bar x+t_ku)-F(\bar x)-t_kF^{\circ}(\bar x, u)]
\\
&\leqq \frac12t^2_kL\|v^k-v\|+[F(y^k) - F(\bar x+t_ku)] +  [F(\bar x+t_ku)-F(\bar x)-t_kF^{\circ}(\bar x, u)]
\end{align*}
for all $k\geqq k_0$. This implies that
\begin{align*}
0&\leqq \lim\limits_{k\to\infty} L\|v^k-v\| + \limsup\limits_{k\to\infty}\dfrac{F(\bar x+t_ku+\frac12 t^2_kv) - F(\bar x+t_ku)}{\frac12t^2_k}
\\
&+\limsup\limits_{k\to\infty}\dfrac{F(\bar x+t_ku)-F(\bar x)-t_kF^{\circ}(\bar x, u)}{\frac12t^2_k}
\\
&\leqq\limsup\limits_{\mathop {x \to \bar x}\limits_{t  \downarrow 0} }\dfrac{F(x+tv) - F(x)}{t} + \limsup\limits_{t\downarrow 0}\dfrac{F(\bar x+tu)-F(\bar x)-tF^{\circ}(\bar x, u)}{\frac12t^2}
\\
&= F^{\circ}(\bar x, v) + F^{\circ\circ} (\bar x, u).
\end{align*}
Therefore, $F^{\circ}(\bar x, v) + F^{\circ\circ} (\bar x, u)\geqq 0$. The proof is complete.
 \end{proof}
\section{Second-order constraint qualification and regularity condition}          
\label{Abadie_RC_sect}
From now on, we consider problem \eqref{problem} under the following assumptions:
\begin{equation*}
\begin{cases}
\text{The functions}\ \ f_i, i\in I, g_j, j\in J(\bar x), \ \ \text{are locally Lipschitz  at} \ \ \bar x,
\\
\text{The functions}\ \  g_j, j\in J\setminus J(\bar x),\ \ \text{are continuous at}\ \  \bar x,
\end{cases}
\end{equation*}
where $\bar x$ is a feasible point of \eqref{problem} and $J(\bar x)$ is the {\em active index set} at $\bar x$, that is,
$$J(\bar x):=\{j\in J\,:\,g_j(\bar x)=0\}.$$
For any vectors $a=(a_1, a_2)$ and $b=(b_1, b_2)$ in $\mathbb{R}^2$, we denote the lexicographic order by
\begin{align*}
a&\leqq_{\rm lex} b,\ \  {\rm iff} \ \ a_1<b_1\ \   {\rm or} \ \  (a_1=b_1\ \ {\rm and }\ \   a_2\leqq b_2),
\\
a&<_{\rm lex} b,\ \  {\rm iff} \ \ a_1<b_1\ \   {\rm or} \ \  (a_1=b_1\ \ {\rm and }\ \   a_2< b_2).
\end{align*}

Let us introduce some notations which are used in the sequel. For each $\bar x\in Q_0$ and $u\in X$, put
\begin{align*}
&Q:=Q_0\cap \{x\in X\,:\, f_i(x)\leqq f_i(\bar x), \ \ i\in I\},
\\
&J(\bar x; u):=\{j\in J(\bar x)\,:\, g_j^{\circ}(\bar{x}, u)=0\},
\\
&I(\bar x; u):=\{i\in I\,:\,f_i^{\circ}(\bar{x}, u)=0\}.
\end{align*}
We say that $u$ is a {\em critical direction} of \eqref{problem}  at $\bar x$   iff
\begin{align*}
f_i^{\circ}(\bar{x}, u)&\leqq 0, \ \ \ \forall i\in I,
\\
f_i^{\circ}(\bar{x}, u)&= 0, \ \ \ \mbox{at least one} \ \ i\in I,
\\
g_j^{\circ}(\bar{x}, u)&\leqq 0, \ \ \ \forall j\in J(\bar x).
\end{align*}
The set of all critical directions of \eqref{problem} at $\bar x$ is denoted by $\mathcal{C}(\bar x)$. Obviously, $0\in\mathcal{C}(\bar x)$.

We now use the following second-order approximation sets for $Q$ and $Q_0$ to introduce second-order constraint qualifications and regularity condition.  For each $\bar x\in Q_0$ and $u\in X$, set
\begin{eqnarray*}
	L^2(Q; \bar x, u)&&:=\bigg\{v\in X\, :\, F^2_i(\bar{x}; u, v)\leqq_{\rm lex} (0, 0), \ \ i\in I
	\\
	&& \qquad\ \ \,\,\text{and} \ \ G^2_j(\bar{x}; u, v)\leqq_{\rm lex} (0,0),\ \ j\in J(\bar x)\bigg\},
	\\
	L^2(Q_0; \bar x, u)&&:=\bigg\{v\in X\, :\, G^2_j(\bar{x}; u, v)\leqq_{\rm lex} (0,0),\ \  j\in J(\bar x)\bigg\},
	\\
	L_0^2(Q_0; \bar x, u)&&:=\bigg\{v\in X\, :\, G^2_j(\bar{x}; u, v)<_{\rm lex} (0,0),\ \  j\in J(\bar x)\bigg\},
\end{eqnarray*}
\begin{eqnarray*}
	A(\bar x; u)&&:=\bigg\{v\in X\,:\, \forall j\in J(\bar x; u)\,\,\exists \delta_j>0 \ \ \mbox{with}\ \ g_j\bigg(\bar x+tu+\frac12t^2v\bigg)\leqq 0
	\\
	&& \qquad \qquad\qquad \qquad\qquad \qquad\qquad \qquad \qquad \qquad\qquad \qquad\,\forall t\in (0,\delta_j)\bigg\},
	\\
	B(\bar x; u)&&:=\bigg\{v\in X\,:\, g_j^{\circ}(\bar{x}, v)+g_j^{\circ\circ}(\bar{x}, u) \leqq 0, \ \ \forall j\in J(\bar x; u)\bigg\},
\end{eqnarray*}
where
\begin{align*}
F^2_i(\bar{x}; u, v)&:= \left(f_i^{\circ}(\bar{x}, u), f_i^{\circ}(\bar{x}, v)+f_i^{\circ\circ}(\bar{x}, u) \right), \ \  i\in I, v\in X,
\\
G^2_j(\bar{x}; u,v)&:=\left(g_j^{\circ}(\bar{x}, u), g_j^{\circ}(\bar{x}, v)+g_j^{\circ\circ}(\bar{x}, u) \right), \ \  j\in J(\bar x), v\in X.
\end{align*}
For brevity, we denote $L(Q; \bar x):=L^2(Q; \bar x, 0)$. It is easily seen that, for each $u\in\mathcal{C}(\bar x)$, we have
\begin{equation*}
L_0^2(Q_0; \bar x, u)=\bigg\{v\in X\,:\, g_j^\circ(\bar x, v)+g_j^{\circ\circ}(\bar x, u)<0, \ \ j\in J(\bar x, u)\bigg\}.
\end{equation*}

\begin{definition}\label{def-SOCQ}{\rm
Let  $\bar x\in Q_0$ and $u\in X$. We say that:
\begin{enumerate} [(i)]
\item The {\em  Zangwill second-order  constraint qualification} holds at $\bar x$ for the direction $u$ iff
\[
B(\bar x; u)\subset \text{cl}\, A(\bar x; u). \tag{$ZSCQ$}\label{SACQ_1}
\]	
\item   The {\em  Abadie second-order  constraint qualification} holds at $\bar x$ for the direction $u$ iff
			\[
			L^{2}(Q_0; \bar x, u)\subset  T^2(Q_0; \bar x, u). \tag{$ASCQ$}\label{SACQ_2}
			\]	
\item   The {\em Mangasarian--Fromovitz second-order  constraint qualification} holds at $\bar x$ for the direction $u$ iff
			\[ L_0^2(Q_0; \bar x, u) \neq \emptyset.\tag{$MFSCQ$} \label{MFSCQ}
			\]			
\item  The {\em weak Abadie second-order  regularity condition} holds at $\bar x$ for the direction $u$ iff
			\[
			L^{2}(Q; \bar x, u)\subset  T^2(Q_0; \bar x, u). \tag{$WASRC$}\label{SACQ_3}
			\]
\end{enumerate}
	}
\end{definition}

The \eqref{SACQ_1} type was first introduced by Ivanov \cite[Definition 3.2]{Ivanov15} for $C^1$ functions. The \eqref{SACQ_2} type was proposed by Aghezzaf and  Hachimi  for \eqref{problem} with $C^2$ data; see \cite[p.40]{Aghezzaf99}. The \eqref{MFSCQ} type was first introduced in \cite{Ben-Tal80} for  $C^2$ scalar optimization problems. The \eqref{SACQ_3} type was used for $C^{1,1}$ vector optimization problems in \cite{Huy163}. For problems with only locally Lipschitz active constraints and objective functions, these conditions are new.

\begin{definition}\label{def3.2}{\rm Let  $\bar x\in Q_0$. We say that the {\em  Zangwill  constraint qualification $(ZCQ)$} (resp., {\em Abadie constraint qualification $(ACQ)$}, {\em  Mangasarian--Fromovitz constraint qualification $(MFCQ)$}, {\em weak Abadie regularity condition} $(WARC)$) holds at $\bar x$ iff the \eqref{SACQ_1} (resp., \eqref{SACQ_2}, \eqref{MFSCQ}, \eqref{SACQ_3}) holds at $\bar x$ for the direction $0$.
	}
\end{definition}

The following result shows that the \eqref{SACQ_3} is weaker than other constraint qualification conditions in Definition \ref{def-SOCQ}.
\begin{proposition}\label{relations-CQ} Let  $\bar x\in Q_0$ and $u\in X$. Then the following implications hold:
\begin{enumerate}[\rm(i)]
	\item  	$ (\, B(\bar x; u)\subset {\rm cl}\, A(\bar x; u) \, ) $ $\Rightarrow$ $ (\,L^{2}(Q_0; \bar x, u)\subset  T^2(Q_0; \bar x, u)\, )$ $\Rightarrow$\\ 	$ (\,L^{2}(Q; \bar x, u)\subset  T^2(Q_0; \bar x, u)\, ), \,\,i.e,$   $$\eqref{SACQ_1}\Rightarrow\eqref{SACQ_2}\Rightarrow\eqref{SACQ_3}.$$

	\item   $  (\,L_0^2(Q_0; \bar x, u) \neq \emptyset \, ) \Rightarrow	(\, L^{2}(Q_0; \bar x, u)\subset  T^2(Q_0; \bar x, u)\, ),\,\,i.e.,$
	$$\eqref{MFSCQ}\Rightarrow\eqref{SACQ_2}.$$

	\item
	 $(\,L_0^2(Q_0; \bar x, 0) \neq \emptyset \, )\,\, \Rightarrow\, (\,L_0^2(Q_0; \bar x, u) \neq \emptyset,\,\,\forall\, u\,\in\mathcal{C}(\bar x)\, ).$
\end{enumerate}
	
\end{proposition}
\begin{proof} (i) Clearly, $L^2(Q; \bar x, u) \subset L^2(Q_0; \bar x, u)$. Thus the second implication of (i) is trivial. We now assume that the \eqref{SACQ_1} holds at $\bar x$ for the direction $u\in X$. Fix $v\in L^{2}(Q_0; \bar x, u)$. Then,
\begin{equation*}
G^2_j(\bar{x}; u, v)\leqq_{\rm lex} (0,0),\ \ \forall j\in J(\bar x).
\end{equation*}
This implies that
\begin{align*}
g_j^{\circ}(\bar{x}, u)&\leqq 0, \ \ \forall j\in J(\bar x),
\\
g_j^{\circ}(\bar{x}, v)+g_j^{\circ\circ}(\bar{x}, u) &\leqq 0, \ \ \forall j\in J(\bar x; u).
\end{align*}
Thus, $v\in B(\bar x; u)$. Since the \eqref{SACQ_1} holds at $\bar x$ for the direction $u$, we have $v\in \mathrm{cl}\, A(\bar x; u)$. Thus there exists a sequence $\{v^k\}\subset A(\bar x; u)$ converging to $v$. Let $\{t_h\}$ be an arbitrary positive sequence converging to $0$. We claim that there is a subsequence $\{t_{h_k}\}\subset \{t_h\}$ such that
\begin{equation*}
\bar x+t_{h_k}u+\frac12t^2_{h_k}v^k\in Q_0, \ \ \forall k\in \mathbb{N}.
\end{equation*}
We will prove this claim by induction on $k$.

In case of $k=1$, let $\{x_h\}$ be a sequence defined by
$$x^h:=\bar x+t_{h}u+\frac12t^2_{h}v^1\ \ \ \text{for all}\ \ h\in \mathbb{N}.$$
Let us consider the following possible cases for $j\in J$.

{\bf Case 1.} $j\notin J(\bar x)$. This means that $g_j(\bar x)<0$. Since $g_j$ is continuous at $\bar x$ and $\lim\limits_{h\to\infty} x^h=\bar x$, there is $H_1\in\mathbb{N}$ such that $g_j\left(x^h\right)<0$ for all $h\geqq H_1$.

{\bf Case 2.} $j\in J(\bar x)\setminus J(\bar x; u)$. This means that
$g_j(\bar x)=0$ and $g_j^{\circ} (\bar x, u)<0$. We claim that there exists $H_2\in\mathbb{N}$ such that $g_j\left(x^h\right)<0$ for all $h\geqq H_2$.
Indeed, if otherwise, there is a subsequence $\{t_{h_l}\}\subset \{t_h\}$ satisfying
$$g_j\left(\bar x+t_{h_l}u+\frac12t^2_{h_l}v^1\right)\geqq g_j(\bar x)=0, \ \ \forall l\in\mathbb{N},$$
or, equivalently,
\begin{equation*}
g_j\left(\bar x+t_{h_l}\left(u+\frac12t_{h_l}v^1\right)\right)\geqq g_j(\bar x), \ \ \forall l\in\mathbb{N}.
\end{equation*}
Clearly, $\lim\limits_{l\to\infty}\left(u+\frac12t_{h_l}v^1\right)=u$.  By Lemma \ref{lemma1}, $g_j^{\circ}(\bar x, u)\geqq 0$, and which contradicts with the fact that $g_j^{\circ} (\bar x, u)<0$.

{\bf Case 3.} $j\in J(\bar x; u)$. Since $v^1\in A(\bar x; u)$ and $j\in J(\bar x; u)$, there exists $\delta_j>0$ such that
$$g_j\left(\bar x+tu+\frac12t^2v^1\right)\leqq 0, \ \ \forall t\in (0, \delta_j).$$
From $\lim\limits_{h\to\infty}t_h=0$ it follows that there is $H_3\in\mathbb{N}$ such that $t_h\in (0, \delta_j)$ for all $h\geqq H_3$. Thus, $g_j\left(x^h\right)\leqq 0$ for all $h\geqq H_3$.

Put $h_1:=\max\{H_1, H_2, H_3\}$. Then, we have $g_j\left(x^h\right)\leqq 0$ for all $h\geqq h_1$ and $j\in J$. This implies that
$$\bar x+t_{h}u+\frac12t^2_{h}v^1\in Q_0 \ \ \forall h\geqq h_1.$$
Thus, by induction on $k$, there exists a subsequence $\{t_{h_k}\}\subset \{t_h\}$ such that
\begin{equation*}
\bar x+t_{h_k}u+\frac12t^2_{h_k}v^k\in Q_0, \ \ \forall k\in \mathbb{N}.
\end{equation*}
From this, $\lim\limits_{k\to\infty}t_{h_k}=0$, and $\lim\limits_{k\to\infty} v^k=v$, it follows that $v\in T^2(Q_0; \bar x, u)$. Since $v$ is arbitrary in $L^2(Q_0; \bar x, u)$, we have
$$L^2(Q_0; \bar x, u)\subset T^2(Q_0; \bar x, u).$$
Thus the \eqref{SACQ_2} holds at $\bar x$ for the direction $u$.

(ii) We now assume that the \eqref{MFSCQ} holds at $\bar x$ for the direction $u\in X$ and $v^0\in L^2_0(Q_0; \bar x, u)$. Fix $v\in L^2(Q_0; \bar x, u)$. Then,
 \begin{align*}
 g_j^{\circ}(\bar{x}, u)&\leqq 0, \ \ \forall j\in J(\bar x),
 \\
 g_j^{\circ}(\bar{x}, v)+g_j^{\circ\circ}(\bar{x}, u) &\leqq 0, \ \ \forall j\in J(\bar x; u).
 \end{align*}
Let $\{s_k\}$ and $\{t_h\}$ be any positive sequences converging to zero. For each $k\in\mathbb{N}$, put $v^k:=s_kv^0+(1-s_k)v$. Then, $\lim\limits_{k\to\infty}v^k=v$. We claim that there exists a subsequence $\{t_{h_k}\}$ of $\{t_h\}$ such that
\begin{equation*}
\bar x+t_{h_k}u+\frac12t^2_{h_k}v^k\in Q_0, \ \ \forall k\in \mathbb{N}.
\end{equation*}
Consequently, $v\in T^2(Q_0; \bar x, u)$ and we therefore get the \eqref{SACQ_2}.

Indeed, for $k=1$, we have that $v^1=s_1v^0+(1-s_1)v$. Fix $j\in J$. If $j\in J\setminus J(\bar x; u)$, then,  we prove as in Case 1 and Case 2 of the proof of assertion (i) that there exists $H_1\in\mathbb{N}$ such that
\begin{equation*}
g_j\left(x^h\right)<0, \ \ \forall h\geqq H_1,
\end{equation*}
where $x^h:=\bar x+t_{h}u+\frac12t^2_{h}v^1$. If $j\in J(\bar x; u)$, then
\begin{equation*}
g_j^{\circ}(\bar{x}, v^0)+g_j^{\circ\circ}(\bar{x}, u) < 0.
\end{equation*}
Hence,
\begin{align*}
g_j^{\circ}(\bar{x}, v^1)+g_j^{\circ\circ}(\bar{x}, u) &\leqq s_1g_j^{\circ}(\bar{x}, v^0)+(1-s_1)g_j^{\circ}(\bar{x}, v)+g_j^{\circ\circ}(\bar{x}, u)
\\
&= s_1[g_j^{\circ}(\bar{x}, v^0)+g_j^{\circ\circ}(\bar{x}, u)]+(1-s_1)[g_j^{\circ}(\bar{x}, v)+g_j^{\circ\circ}(\bar{x}, u)]
\\
&<0.
\end{align*}
Thus,
\begin{align*}
\limsup_{h\to\infty}\frac{g_j(x^h)}{\frac{1}{2}t_h^2}&=\limsup_{h\to\infty}\frac{g_j(x^h)-g_j(\bar x)-t_hg_j^\circ(\bar x; u)}{\frac{1}{2}t_h^2}
\\
&\leqq \limsup_{h\to\infty}\frac{g_j((\bar x+t_hu)+\frac{1}{2}t_h^2v^1)-g_j(\bar x+t_hu)}{\frac{1}{2}t_h^2}
\\
&+\limsup_{h\to\infty}\frac{g_j(\bar x+t_hu)-g_j(\bar x)-t_hg_j^\circ(\bar x; u)}{\frac{1}{2}t_h^2}
\\
&\leqq g_j^\circ(\bar x; v^1)+g_j^{\circ\circ}(\bar x; u)
\\
&<0.
\end{align*}
This implies that there exists $H_2\in\mathbb{N}$ such that $g_j(x^h)<0$ for all $h\geqq H_2$. Put $h_1:=\max\{H_1, H_2\}$. Then we have $g_j(x^h)<0$ for all $h\geqq h_1$ and $j\in J$. Thus,
$$\bar x+t_hu+\frac{1}{2}t_h^2v^1\in Q_0\ \ \forall h\geqq h_1,$$
and the assertion follows by induction on $k$.

(iii)  Assume that there exists $v^0\in L_0^2(Q_0; \bar x, 0)$. Then $g_j^\circ(\bar x, v^0)<0$ for all $j\in J(\bar x)$. Let $u\,\in\mathcal{C}(\bar x)$. For each $t>0$, put $v(t):=u+tv^0$. We claim that there exists $t>0$ such that $v(t)\in L_0^2(Q_0; \bar x, u)$. Indeed, for each $j\in J(\bar x; u)$, one has
\begin{align*}
g^\circ_j(\bar x, v(t))+g_j^{\circ\circ}(\bar x, u)&\leqq g^\circ_j(\bar x, u)+tg_j^\circ(\bar x, v^0)+g_j^{\circ\circ}(\bar x, u)
\\
&=tg_j^\circ(\bar x, v^0)+g_j^{\circ\circ}(\bar x, u)
\\
&<0
\end{align*}
for $t$ large enough. This implies that $v(t)\in L_0^2(Q_0; \bar x, u)$ for $t$ large enough, as required.		
\end{proof}

The relations between second-order constraint qualifications are summarized in Figure \ref{Fig1}.
\begin{center}
	\begin{figure}[htp]
		\begin{center}
			\includegraphics[height=4cm,width=7cm]{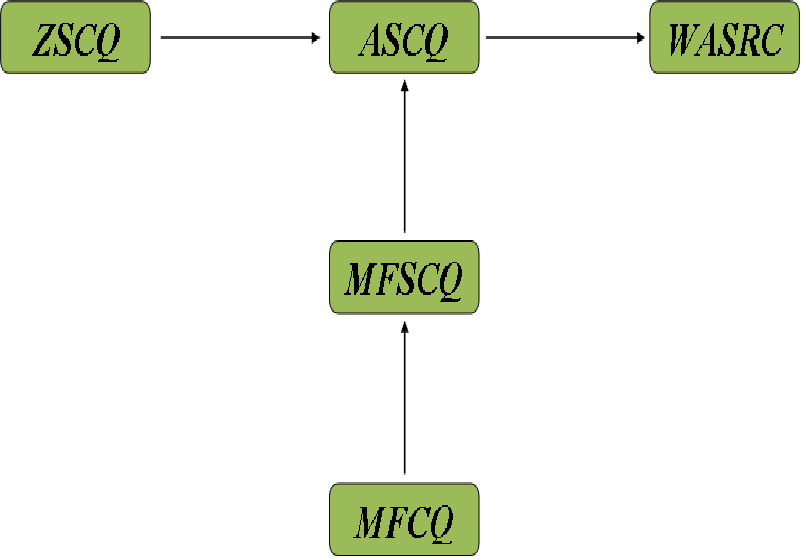}
		\end{center}
		\caption{Relations between second-order constraint qualifications}
		\label{Fig1}
	\end{figure}
\end{center}

\begin{remark}
{\rm
The forthcoming Examples~\ref{ex4.1} and \ref{ex4.2} show that  $\eqref{SACQ_3} \not\Rightarrow\eqref{SACQ_1}$ and  $\eqref{SACQ_3} \not\Rightarrow\eqref{MFSCQ}$.
}
\end{remark}

For the remainder of this paper, we apply the \eqref{SACQ_3} to establish some second-order $KKT$ necessary optimality conditions for efficient solutions of   \eqref{problem}.  We point out that, by Proposition~\ref{relations-CQ},  these results  
still valid when the \eqref{SACQ_3}
is replaced by one of \eqref{SACQ_1}, \eqref{SACQ_2} and \eqref{MFSCQ}.

\section{Second-order optimality conditions for efficiencies}
\label{Second_order_optim_sect}

In this section, we apply the \eqref{SACQ_3} to establish some second-order $KKT$ necessary optimality conditions in primal form for local (weak) efficient solutions of   \eqref{problem}.

The following theorem gives a first-order necessary optimality condition for  \eqref{problem} under the reqularity condition ($WARC$).

\begin{theorem}\label{first_order_nec}
	 If $\bar x\in Q_0$ is a local {\rm(}weak{\rm)} efficient solution of \eqref{problem} and $(WARC)$ holds at $\bar x$, then the system
	\begin{align}
	f_i^{\circ} (\bar x, u)&<0, \ \ i\in I, \label{equa:3}
	\\
	g_j^{\circ}(\bar x, u)&\leqq 0, \ \ j\in J(\bar x),\label{equa:4}
	\end{align}
	has no solution $u\in X$.
\end{theorem}

\begin{proof}
Arguing by contradiction, assume that there exists $u\in X$ satisfying conditions  \eqref{equa:3} and \eqref{equa:4}. This implies that $u\in L(Q; \bar x)$.  Since the $(WARC)$ holds at $\bar x$, one has
$$L(Q; \bar x)\subset T(Q_0; \bar x).$$
Consequently, $u\in T(Q_0; \bar x)$. Thus there exist $t_k\to 0^+$ and $u^k\to u$ such that
$$\bar x+t_ku^k\in Q_0$$
for all $k\in\mathbb{N}$. We claim that, for each $i\in I$, there exists $K_i\in\mathbb{N}$ satisfying
$$f_i(\bar x+t_ku^k)<f_i(\bar x), \ \ \forall k\geqq K_i.$$
Indeed, if otherwise, there exist $i\in I$ and a sequence $\{k_l\}\subset \mathbb{N}$ such that
$$f_i(\bar x+t_{k_l}u^{k_l})\geqq f_i(\bar x), \ \ \forall l\in\mathbb{N}. $$
By Lemma \ref{lemma1}, we have  $f_i^{\circ}(\bar x, u)\geqq 0$,  contrary to \eqref{equa:3}.

Put $K_0:=\max\, \{K_1, \ldots, K_p\}$. Then,
$$f_i((\bar x+t_ku^k)<f_i(\bar x)$$
for all $k\geqq K_0$ and $i\in I$, which contradicts the hypothesis of the theorem.
\end{proof}
\begin{remark}{\rm
		\begin{enumerate}[(i)]
			\item  Recently, Gupta et al. \cite[Theorems 3.1]{Gupta2017} showed that {\em``If $\bar x$ is an efficient solution of \eqref{problem}, $X=\mathbb{R}^n$, for each $i\in I$, $f_i$ is $\partial^c$-quasiconcave at $\bar x$, and there exists $i\in I$ such that
				\begin{equation}\label{Gupta-condition}
				L(M^i; \bar x)\subset T(M^i; \bar x),
				\end{equation}
				where
				\begin{align*}
				M^i&:=\{x\in Q_0\;:\; f_i(x)\leqq f_i(\bar x)\},
				\\
				L(M^i; \bar x)&:=\{u\in X\;:\; f^{\circ}_i(\bar x; u)\leqq 0, g^{\circ}_j(\bar x; u)\leqq 0, j\in J(\bar x)\},
				\end{align*}	
				then the system \eqref{equa:3}--\eqref{equa:4} has no solution''}.
			
			Clearly,
			\begin{align*}
			T(M^i; \bar x)&\subset T(Q_0; \bar x),
			\\
			L(Q; \bar x)&\subset L(M^i; \bar x).
			\end{align*}
			This implies that if condition \eqref{Gupta-condition} holds at $\bar x$, then so does the $(WARC)$. Thus, Theorem \ref{first_order_nec} improves \cite[Theorems 3.1]{Gupta2017}. We note here that the assumption that  $f_i$ is $\partial^c$-quasiconcave at $\bar x$ is not necessary in our result.
			\item  Theorem \ref{first_order_nec} also improves \cite[Theorems 3.3]{Gupta2017}. Theorem 3.3 in \cite{Gupta2017} is as follows: {\em ``If $\bar x$ is a weak efficient solution of \eqref{problem}, $X=\mathbb{R}^n$, $Q_0$ is convex, for each $i\in I$, $f_i$ is $\partial^c$-quasiconcave at $\bar x$, and there exists $i\in I$ such that
				\begin{equation}\label{Gupta-condition-ii}
				L(M^i; \bar x)\subset \mathrm{cl}\,\mathrm{conv}\, T(M^i; \bar x),
				\end{equation} 		
				then the system \eqref{equa:3}--\eqref{equa:4} has no solution''}.
			
			Since $T(M^i; \bar x)\subset T(Q_0; \bar x)$ and $Q_0$ is a closed convex set, we have
			$$\mathrm{cl}\,\mathrm{conv}\, T(M^i; \bar x) \subset T(Q_0; \bar x).$$
			This implies the $(WARC)$ is weaker than condition \eqref{Gupta-condition-ii} and so Theorem \ref{first_order_nec} sharpens  \cite[Theorems 3.3]{Gupta2017}. We would like to remark that our result does not require any convexity assumptions.
		\end{enumerate}		
		
	}
\end{remark}

Now we are ready to present our result of second-order $KKT$ optimality conditions for local (weak) efficient solutions of \eqref{problem} under the \eqref{SACQ_3}.
\begin{theorem}\label{nec_condition_weak_eff} Let $\bar x$ be a local {\rm(}weak{\rm)}  efficient solution of \eqref{problem}. Suppose that the (\ref{SACQ_3}) holds at $\bar x$ for any critical direction. Then, the  system
	\begin{align}
	F^2_i(\bar{x}; u, v)&<_{\rm lex} (0, 0),\ \ \ i\in I, \label{equa:5}
	\\
	G^2_j(\bar{x}; u, v)&\leqq_{\rm lex} (0, 0),\ \ \  j\in J(\bar x)\label{equa:6}.
	\end{align}
	has no solution $(u, v)\in X\times X$.
\end{theorem}
\begin{proof} Arguing by contradiction, assume that there exists $(u, v)\in X\times X$ satisfying conditions \eqref{equa:5} and \eqref{equa:6}. It follows that $v\in L^2 (Q; \bar x, u)$ and
\begin{eqnarray*}
	f_i^{\circ}(\bar x, u)&\leqq 0, \ \ \ &i\in I,
	\\
	g_j^{\circ}(\bar x, u)&\leqq 0, \ \ \ &j\in J (\bar x).
\end{eqnarray*}
Since the \eqref{SACQ_3} holds at $\bar x$, so does the $(WARC)$. By Theorem \ref{first_order_nec}, there exists $i\in I$ such that $f_i^{\circ}(\bar x, u)=0$. This means that $u$ is a critical direction of \eqref{problem} at $\bar x$. Since the \eqref{SACQ_3} holds at $\bar x$ for the critical direction $u$, we have
$$v\in T^2(Q_0; \bar x, u).$$
Thus there exist a sequence  $\{v^k\}$ converging to $v$ and a positive sequence  $\{t_k\}$ converging to $0$ such that
$$x^k:=\bar x+t_ku+\frac12t_k^2v^k\in Q_0,\ \ \ \forall k\in\mathbb{N}.$$
We claim that, for each $i\in I$, there exists $K_i\in \mathbb{N}$ such that
$$f_i(x^k)<f_i(\bar x)$$
for all $k\geqq K_i$. Indeed, if otherwise, there exist $i_0\in I$ and a sequence $\{k_l\}\subset \mathbb{N}$ satisfying
\begin{equation}\label{equa:7}
f_{i_0}\left(\bar x+t_{k_l}u+\frac12t^2_{k_l}v^{k_l}\right)\geqq f_{i_0}(\bar x), \ \ \forall l\in\mathbb{N}.
\end{equation}
We consider the following possible cases for $i_0$.

{\bf Case 1.} $i_0\in I(\bar x; u)$. This means that  $f_{i_0}^{\circ}(\bar x, u)=0$. From \eqref{equa:5} it follows that
\begin{equation}\label{equa:8}
f_{i_0}^{\circ}(\bar x, v)+f_{i_0}^{\circ\circ}(\bar x, u)<0.
\end{equation}
From \eqref{equa:7}, $\lim\limits_{l\to\infty} t_{k_l}=0$, $\lim\limits_{l\to\infty} v^{k_l}=v$, and Lemma \ref{lemma2}, it follows that
$$f_i^{\circ}(\bar x, v)+f_i^{\circ\circ}(\bar x, u)\geqq 0,$$
contrary to \eqref{equa:8}.

{\bf Case 2.} $i_0\notin I(\bar x; u)$. This means that $f_{i_0}^{\circ}(\bar x, u)<0$.  In this case we now rewrite \eqref{equa:7} as
$$f_{i_0}\left(\bar x+t_{k_l}\left(u+\frac12t_{k_l}v^{k_l}\right)\right)\geqq f_{i_0}(\bar x), \ \ \forall l\in\mathbb{N}.$$
From $\lim\limits_{l\to\infty} t_{k_l}=0$, $\lim\limits_{l\to\infty} \left(u+\frac12t_{k_l}v^{k_l}\right)=u$, and Lemma \ref{lemma1}, it follows that $f_{i_0}^{\circ}(\bar x, u)\geqq 0$. This contradicts the fact that $f_{i_0}^{\circ}(\bar x, u)<0$.

Put $K_0:=\max\{K_i\,:\, i\in I\}$. Then, we have
$$f_i(x^k)<f_i(\bar x)$$
for all $k\geqq K_0$ and $i\in I$, which contradicts the hypothesis of the theorem.
\end{proof}

An immediate consequence of the above theorem is the following corollary.
\begin{corollary}\label{second_order_nec}  Let $\bar x$ be a local {\rm(}weak{\rm)}  efficient solution  of \eqref{problem} and $u\in \mathcal{C}(\bar x)$. Suppose that the \eqref{SACQ_3} holds at $\bar x$ for the direction $u$. Then the following system
	\begin{align}
	&f_{i}^{\circ}(\bar x, v)+f_{i}^{\circ\circ}(\bar x, u)<0, \ \ i\in I(\bar x; u), \label{equa:9}
	\\
	&g_{j}^{\circ}(\bar x, v)+g_{j}^{\circ\circ}(\bar x, u)\leqq 0, \ \ j\in J(\bar x, u), \label{equa:10}
	\end{align}
	has no solution $v\in X$.
\end{corollary}
\begin{remark}{\rm Suppose that $F\colon X\to \mathbb{R}$ is of class $C^1(X)$, i.e., $F$ is Fr\'echet differentiable and its gradient mapping is continuous on $X$. If $F$ is second-order directionally differentiable at $\bar x$, i.e., there exists
		$$F^{\prime\prime}(\bar x, u):= \lim\limits_{t\downarrow 0} \dfrac{F(\bar x+tu)-F(\bar x)- t\langle \nabla F(\bar x), u\rangle}{\frac12t^2},\ \ u\in X,$$
		then $F^{\prime\prime}(\bar x, u) = F^{\circ\circ}(\bar x, u)$ for all $u\in X$. In \cite{Ivanov15}, Ivanov considered problem \eqref{problem} under the following conditions:
		\begin{equation}\label{Ivanov_condition}\tag{$\mathfrak{C}$}
		\left.
		\begin{aligned}
		&\text{The functions } g_j, j\notin J(\bar x) \text{ are continuous at } \bar x;
		\\
		&\text{The functions }  f_i, i\in I, g_j, j\in J(\bar x) \text{ are of class } C^1(X);
		\\
		&\text{If } \langle \nabla f_i(\bar x), u\rangle=0, \text{ then there exists } f_i^{\prime\prime} (\bar x, u);
		\\
		&\text{If } \langle \nabla g_j(\bar x), u\rangle=0, j\in J(\bar x), \text{ then there exists } g_j^{\prime\prime} (\bar x, u).
		\end{aligned}
		\right \}
		\end{equation}
		If condition \eqref{Ivanov_condition} holds at $\bar x$ for the direction $u$, then the system \eqref{equa:9}--\eqref{equa:10} becomes
		\begin{align*}
		&\langle \nabla f_i(\bar x), v\rangle+f_i^{\prime\prime}(\bar x, u)<0, \ \ i\in I(\bar x, u),\\
		&\langle \nabla g_j(\bar x), v\rangle+g_j^{\prime\prime}(\bar x, u)\leqq 0, \ \ j\in J(\bar x, u).
		\end{align*}
		Since the \eqref{SACQ_3} is weaker than the \eqref{SACQ_1}, Corollary \ref{second_order_nec} improves and extends result of Ivanov \cite[Theorem 4.1]{Ivanov15} and  of Huy et al. \cite[Theorem 3.2]{Huy163}. To illustrate, we consider the following example.
	}
\end{remark}
\begin{example}\label{ex4.1}
{\rm Let $f\colon \mathbb{R}^2 \to\mathbb{R}^3$ and $g\colon \mathbb{R}^2\to \mathbb{R}$ be two maps defined by
		\begin{align*}
		f(x)&:= (f_1(x), f_2(x), f_3(x))=(x_2, x_1+x_2^2, -x_1-x_1|x_1|+x_2^2)\\
		g(x)&:=|x_1|+x_2^3-x_1^2, \ \ \forall x=(x_1, x_2)\in\mathbb{R}^2.
		\end{align*}
		Then the feasible set of \eqref{problem} is
		$$Q_0=\{(x_1, x_2)\in\mathbb{R}^2\,:\,|x_1|+x_2^3-x_1^2\leqq 0\}.$$
		Let $\bar x=(0,0)\in Q_0$. It is easy to check that $\bar x$ is an efficient solution of \eqref{problem}. For each $u=(u_1, u_2)\in\mathbb{R}^2$, we have
		\begin{align*}
		&f_1^{\circ}(\bar x, u)=\langle\nabla f_1(\bar x), u\rangle=u_2, f_2^{\circ}(\bar x, u)=\langle\nabla f_2(\bar x), u\rangle= u_1
		\\
		&f_3^{\circ}(\bar x, u)=\langle\nabla f_3(\bar x), u\rangle=-u_1, g^{\circ}(\bar x, u)=|u_1|.
		\end{align*}
		Thus,
		$$\mathcal{C}(\bar x)=\{(u_1, u_2)\in\mathbb{R}^2\,:\, u_1=0, u_2\leqq 0\}.$$
		Clearly, $0_{\mathbb{R}^2}:=(0,0)$ is a critical direction at $\bar x$. We claim that the \eqref{SACQ_3} holds at $\bar x$ for the direction $0_{\mathbb{R}^2}$. Indeed, we have
		\begin{equation*}
		L^2(Q; \bar x, 0_{\mathbb{R}^2})=\{(v_1, v_2)\in\mathbb{R}^2\,:\, v_1=0,  v_2\leqq 0\}.
		\end{equation*}
		An easy computation shows  that
		$$T^2(Q_0; \bar x, 0_{\mathbb{R}^2})=T(Q_0; \bar x)=\{(v_1, v_2)\in\mathbb{R}^2\,:\, v_1=0,  v_2\leqq 0\}.$$ This implies that the \eqref{SACQ_3} holds at $\bar x$ for the direction $0_{\mathbb{R}^2}$. By Corollary \ref{second_order_nec}, the system
		\begin{align*}
		&f_{i}^{\circ}(\bar x, v)+f_{i}^{\circ\circ}(\bar x, 0_{\mathbb{R}^2})<0, \ \ i\in I(\bar x; 0_{\mathbb{R}^2}),
		\\
		&g^{\circ}(\bar x, v)+g^{\circ\circ}(\bar x, 0_{\mathbb{R}^2})\leqq 0,
		\end{align*}
		has no solution $v\in\mathbb{R}^2$. The second-order necessary conditions of Huy et al. \cite[Theorem 3.2]{Huy163} and of Ivanov \cite[Theorem 4.1]{Ivanov15} are not applicable to this example as the  constraint function $g$ is not Fr\'echet differentiable at $\bar x$. Furthermore, the \eqref{SACQ_1} does not hold at $\bar x$ for the direction $0_{\mathbb{R}^2}$. Indeed, we have
		$$B(\bar x; 0_{\mathbb{R}^2})=\{(v_1, v_2)\in\mathbb{R}^2\,:\, v_1=0, v_2\in\mathbb{R}\}.$$
		Let $v=(v_1, v_2)\in\mathbb{R}^2$. We have $v\in A(\bar x; 0_{\mathbb{R}^2})$ if and only if
		there exists $\delta>0$ such that
		$$g\left(\bar x+t0_{\mathbb{R}^2}+\frac12t^2v\right)\leqq 0, \ \ \forall t\in(0,\delta),$$
		or, equivalently,
		\begin{equation}\label{equa:11}
		|v_1|-\frac12t^2v_1^2+\frac14t^4v_2^3\leqq 0, \ \ \forall t\in(0,\delta).
		\end{equation}
		It is easy to check that \eqref{equa:11} is true if and only if $v_1=0$ and $v_2\leqq 0$. Thus,
		$$A(\bar x; 0_{\mathbb{R}^2})=\{(v_1, v_2)\in\mathbb{R}^2\,:\, v_1=0, v_2\leqq 0\}.$$
		Clearly, $B(\bar x; 0_{\mathbb{R}^2}) \nsubseteq \mbox{cl}\,A(\bar x; 0_{\mathbb{R}^2})$. This means that the \eqref{SACQ_1} does not hold at $\bar x$ for the direction $0_{\mathbb{R}^2}$.
	}
\end{example}
\begin{remark}{\rm Recently, by using the \eqref{MFSCQ}, Luu \cite[Corollary 5.2]{Luu17} derived some second-order KKT necessary conditions for weak efficient solutions of differentiable vector problems in terms of the second-order upper generalized directional derivatives.  By Proposition \ref{relations-CQ}, the \eqref{SACQ_3} is weaker than the \eqref{MFSCQ}.
Thus, Corollary \ref{second_order_nec} improves  \cite[Corollary 5.2]{Luu17}.
To see this, let us consider the following example.
}
\end{remark}
\begin{example} \label{ex4.2}
Let $f\colon \mathbb{R}^2 \to\mathbb{R}^2$ and $g\colon \mathbb{R}^2\to \mathbb{R}^2$ be two maps defined by
	\begin{align*}
	f(x)&:= (f_1(x), f_2(x))=(x_1+x_2^2, -x_1-x_1|x_1|+x_2^2)\\
	g(x)&:=(g_1(x), g_2(x))=(x_1-x_2^2, -x_1-x_2^2), \ \ \forall x=(x_1, x_2)\in\mathbb{R}^2.
	\end{align*}
	Then the feasible set of \eqref{problem} is
	$$Q_0=\{(x_1, x_2)\in\mathbb{R}^2\,:\, -x_2^2\leqq x_1\leqq x^2_2\}.$$
	Let $\bar x=(0,0)\in Q_0$. Clearly,  $\bar x$ is an efficient solution of \eqref{problem}. It is easy to check that the \eqref{SACQ_3} holds at $\bar x$ for the critical direction $0_{\mathbb{R}^2}$ but not the \eqref{MFSCQ}. Thus Corollary \ref{second_order_nec} can be applied for this example, but not \cite[Corollary 5.2]{Luu17}.
	
\end{example}

\section{Strong second-order optimality condition for local Geoffrion properly efficiencies}\label{Strong_Second_order_optim_sect}

In this section, we apply the \eqref{SACQ_3} to establish a strong second-order $KKT$ necessary optimality condition for  a local Geoffrion properly efficient solution  of \eqref{problem}.
\begin{theorem}\label{Geoffrion_necessary_I} Let $\bar x\in Q_0$ be a local  Geoffrion properly efficient solution of \eqref{problem}. Suppose that the \eqref{SACQ_3} holds at $\bar x$ for any critical direction. Then the system
	\begin{eqnarray}
	F^2_i(\bar{x}; u, v)&\leqq_{\rm lex} (0, 0),\ \ \ &i\in I, \label{equ:G1}
	\\
	F^2_i(\bar{x}; u, v)&<_{\rm lex} (0, 0),\ \ \ &\mbox{at least one} \ \ i\in I(\bar x; u), \label{equ:G2}
	\\
	G^2_j(\bar{x}; u, v)&\leqq_{\rm lex} (0, 0),\ \ \  &j\in J(\bar x) \label{equ:G3}
	\end{eqnarray}
	has no solution $(u, v)\in X\times X$.		
\end{theorem}
\begin{proof} Arguing by contradiction, assume that the system \eqref{equ:G1}--\eqref{equ:G3} admits a solution $(u, v)\in X\times X$. Without any loss of generality we may assume that
\begin{equation*}
F^2_1(\bar{x}; u, v)<_{\rm lex} (0, 0),
\end{equation*}
where $1\in I(\bar x; u)$. This implies that
\begin{equation}\label{equ:G4}
f_{1}^{\circ}(\bar x, v)+f_{1}^{\circ\circ}(\bar x, u)<0.
\end{equation}
From \eqref{equ:G1} and \eqref{equ:G3} it follows that $v\in L^2 (Q; \bar x, u)$ and
\begin{eqnarray*}
	f_{i}^{\circ}(\bar x, u)&\leqq 0, \ \ \ &i\in I,
	\\
	g_{j}^{\circ}(\bar x, u)&\leqq 0, \ \ \ &j\in J (\bar x).
\end{eqnarray*}
This and $1\in I(\bar x; u)$ imply that $u$ is a critical direction at $\bar x$. Since the \eqref{SACQ_3} holds at $\bar x$ for the critical direction $u$, we have $v\in T^2(Q_0; \bar x, u).$ Thus there exist a sequence  $\{v^k\}$ converging to $v$ and a positive sequence  $\{t_k\}$ converging to $0$ such that
$$x^k:=\bar x+t_ku+\frac12t_k^2v^k\in Q_0,\ \ \ \forall k\in\mathbb{N}.$$

Since $1\in I(\bar x; u)$ and \eqref{equ:G4}, as in the proof of Case 1 of Theorem \ref{nec_condition_weak_eff}, there exists $K_1\in \mathbb{N}$ such that
\begin{equation*}
f_1(x^k)<f_1(\bar x)
\end{equation*}
for all $k\geqq K_1$.

For each $i\in I\setminus I(\bar x; u)$, we have $f_{i}^{\circ}(\bar x, u)<0.$ As in the proof of Case 2 of Theorem \ref{nec_condition_weak_eff}, there exists $K_i\in \mathbb{N}$ such that
\begin{equation*}
f_i(x^k)<f_i(\bar x)
\end{equation*}
for all  $k\geqq K_i$.  Without any loss of generality we may assume that
\begin{equation*}
f_i(x^k)<f_i(\bar x)
\end{equation*}
for all $k\in \mathbb{N}$ and $i\in \{1\}\cup [I\setminus I(\bar x; u)]$. For each $k\in \mathbb{N}$, put
$$I_k:= \{i\in I(\bar x;u)\setminus\{1\} \ : f_i(x^k)>f_i(\bar x)\}.$$
We claim that $I_k$ is nonempty for all $k\in\mathbb{N}$. Indeed, if $I_k=\emptyset$ for some $k\in \mathbb{N}$, then we have
$$f_i(x^k)\leqq f_i(\bar x)\ \ \forall i\in  I(\bar x;u)\setminus\{1\}.$$
Using also the fact that  $f_i(x^k)<f_i(\bar x)$ for all $i\in \{1\}\cup [I\setminus I(\bar x; u)]$, we arrive at a contradiction with the efficiency of $\bar x$.

Since $I_k\subset I(\bar x;u)\setminus\{1\} $ for all $k\in\mathbb{N}$, without any loss of generality, we may assume  that $I_k=\bar I$ is constant for all   $k\in \mathbb{N}$. Thus, for each $i\in \bar I$, we have
\begin{equation*}
f_i(x^k)>f_i(\bar x), \ \ \forall k\in\mathbb{N}.
\end{equation*}
By Lemma \ref{lemma2}, we have
\begin{equation*}
f_i^{\circ}(\bar x, v)+f_i^{\circ\circ}(\bar x, u)\geqq 0,  \ \ i\in \bar I.
\end{equation*}
Since \eqref{equ:G1}, for each $i\in \bar I\subset I(\bar x;u)\setminus\{1\} $, we have
\begin{equation*}
f_i^{\circ}(\bar x, v)+f_i^{\circ\circ}(\bar x, u)\leqq 0.
\end{equation*}
Thus,
\begin{equation}\label{equ:G8}
f_i^{\circ}(\bar x, v)+f_i^{\circ\circ}(\bar x, u)=0, \ \ i\in \bar I.
\end{equation}
Let $\delta$ be a real number satisfying
\begin{equation*}
f_{1}^{\circ}(\bar x, v)+f_{1}^{\circ\circ}(\bar x, u)<\delta<0,
\end{equation*}
or, equivalently,
\begin{equation*}
-[f_{1}^{\circ}(\bar x, v)+f_{1}^{\circ\circ}(\bar x, u)]>-\delta>0.
\end{equation*}
It is easily seen that
$$\limsup_{k\to\infty} \dfrac{f_1(x^k)-f_1(\bar x)}{\frac12t^2_k}\leqq f_{1}^{\circ}(\bar x, v)+f_{1}^{\circ\circ}(\bar x, u).$$
Thus there exists $k_0\in\mathbb{N}$ such that
\begin{equation*}
f_1(\bar x)- f_1(x^k)>-\frac12\delta t^2_k>0
\end{equation*}
for all $k\geqq k_0$. Then, for any $i\in \bar I$ and $k\geqq k_0$, we have
\begin{equation*}
0< \dfrac{f_i(x^k)-f_i(\bar x)}{f_1(\bar x)- f_1(x^k)}\leqq \dfrac{f_i(x^k)-f_i(\bar x)}{-\frac12\delta t^2_k}.
\end{equation*}
From this and \eqref{equ:G8}, we have
\begin{align*}
0\leqq \lim_{k\to\infty}\dfrac{f_i(x^k)-f_i(\bar x)}{f_1(\bar x)- f_1(x^k)}&\leqq \limsup_{k\to\infty}\dfrac{f_i(x^k)-f_i(\bar x)}{-\frac12\delta t^2_k}
\\
&\leqq \limsup_{k\to\infty}\frac{f_i(x^k)-f_i(\bar x+t_ku)}{-\frac12\delta t^2_k}
\\
&+\limsup_{k\to\infty}\frac{f_i(\bar x+t_ku)-f_i(\bar x)-t_kf_i^\circ(\bar x; u)}{-\frac12\delta t^2_k}
\\
&\leqq-\frac{1}{\delta}[f_i^{\circ}(\bar x, v)+f_i^{\circ\circ}(\bar x, u)]
\\
&=0.
\end{align*}
Thus,
$$\lim_{k\to\infty}\dfrac{f_1(x^k)-f_1(\bar x)}{f_i(\bar x)-f_i(x^k)}=+\infty,$$
contrary to the fact that $\bar x$ is a local Geoffrion properly efficient solution of \eqref{problem}. The proof is complete.
\end{proof}

The following corollary is immediate from Theorem \ref{Geoffrion_necessary_I}.
\begin{corollary}\label{Geoffrion_necessary_II}  Let $\bar x\in Q_0$ be a local  Geoffrion properly efficient solution of \eqref{problem} and $u\in \mathcal{C}(\bar x)$. Suppose that the \eqref{SACQ_3} holds at $\bar x$ for the direction $u$. Then the system
	\begin{align*}
	&f_{i}^{\circ}(\bar x, v)+f_{i}^{\circ\circ}(\bar x, u)\leqq 0, \ \ i\in I(\bar x; u),
	\\
	&f_{i}^{\circ}(\bar x, v)+f_{i}^{\circ\circ}(\bar x, u)< 0, \ \ \mbox{at leats one} \ \ i\in I(\bar x; u),
	\\
	&g_{j}^{\circ}(\bar x, v)+g_{j}^{\circ\circ}(\bar x, u)\leqq 0, \ \ j\in J(\bar x, u),
	\end{align*}
	has no solution $v\in X$.
\end{corollary}

The next corollary shows that if the $(WARC)$ holds at $\bar x$, then every Geoffrion properly efficient solution of \eqref{problem} is also proper in the sense of Kuhn and Tucker \cite{Kuhn50}.
\begin{corollary}\label{first_order_nec_cond}  Let $\bar x\in Q_0$ be a local  Geoffrion properly efficient solution of \eqref{problem}. Suppose that the $(WARC)$ holds at $\bar x$. Then the  system
	\begin{align}
	&f_{i}^{\circ}(\bar x, u)\leqq 0, \ \ i\in I,\label{first1}
	\\
	&f_{i}^{\circ}(\bar x, u)< 0, \ \ \mbox{at leats one} \ \ i\in I,\label{first2}
	\\
	&g_{j}^{\circ}(\bar x, u)\leqq 0, \ \ j\in J(\bar x),\label{first3}
	\end{align}
	has no solution $u\in X$.
\end{corollary}
\begin{proof} Since the $(WARC)$ holds at $\bar x$, the \eqref{SACQ_3} holds at $\bar x$ for the critical direction $0$. Clearly, $I(\bar x; 0)=I$ and $J(\bar x; 0)=J(\bar x)$. Thus, applying Corollary \ref{Geoffrion_necessary_II}, the system \eqref{first1}--\eqref{first3} has no solution $u\in X$.
\end{proof}
\begin{remark}{\rm Conditions \eqref{first1}--\eqref{first3} are  also known as strong first-order $KKT$  ($SFKKT$) necessary conditions in primal form. In \cite{Rizvi12}, Burachik et al. introduced a generalized Abadie regularity condition $(GARC)$ and established $SFKKT$   necessary conditions for Geoffrion properly efficient solutions of  differentiable vector optimization problems. Later on, Zhao \cite{Zhao15} proposed an extended generalized Abadie regularity condition $(EGARC)$ and then obtained $SFKKT$  necessary conditions for problems with locally Lipschitz data in terms of Clarke's directional derivatives.  Recall  that the  $(EGARC)$ holds at $\bar x\in Q_0$ if
		\begin{equation}\label{EGARC}
		L(Q; \bar x)\subset \bigcap_{i=1}^l T(M^i; \bar x),
		\end{equation}
		for all $i\in I$; see \cite[Definition 3.1]{Zhao15}. If $f_i$ and $g_j$ are of class $C^1(X)$, then condition \eqref{EGARC}  is called by the generalized Abadie regularity condition $(GARC)$; see \cite[p.483]{Rizvi12}. By the isotony of $T(\,\cdot\,; \bar x)$ and the fact that $M^i\subset Q_0$, we have
		$$T(M^i; \bar x)\subset T(Q_0; \bar x)\ \ \ \text{for all}\ \  i\in I.$$
		Thus the $(WARC)$ is weaker than the $(EGARC)$ $((GARC))$. The following example illustrates our results in which the condition $(WARC)$ is satisfied, but the condition $(EGARC)$ $((GARC))$ is not fulfilled.
It turns out that   Corollary~\ref{first_order_nec_cond} improves and
		extends results of Zhao \cite[Theorem 4.1]{Zhao15} and  Burachik et al. \cite[Theorem 4.3]{Rizvi12}.
	}
\end{remark}
\begin{example}{\rm  Consider the following problem:
		\begin{align*}
		& \text{min}\, f(x):=(f_1(x), f_2(x))
		\\
		&\text{subject to}\ \ x\in Q_0:=\{x\in\mathbb{R}^2\,|\, g(x)\leqq 0\},
		\end{align*}
		where
		$$f_1(x):=|x_1|+x_2^2, f_2(x):=-f_1(x), g(x):=x_2 \ \ \text{for all} \ \ x=(x_1, x_2)\in\mathbb{R}^2.$$
		Clearly, $\bar x=(0,0)$ is a Geoffrion properly efficient solution. The optimality conditions of Burachik et al. \cite[Theorem 4.3]{Rizvi12} cannot be used for this problem as the functions $f_1$ and $f_2$ are not differentiable at $\bar x$.
		
		For each $u=(u_1,u_2)\in\mathbb{R}^2$, we have
		$$f_1^{\circ}(\bar x,u)=|u_1|, f_2^{\circ}(\bar x,u)=-|u_1|, g^{\circ}(\bar x,u)=\langle\nabla g(\bar x), u\rangle=u_2.$$
		It is easy to check that
		$$\mathcal{C}(\bar x)=L(Q; \bar x)=\{(u_1, u_2)\in\mathbb{R}^2\,:\, u_1=0, u_2\leqq 0\}.$$
		We claim that the $(EGARC)$ does not hold at $\bar x$. Indeed, since
		\begin{align*}
		M^1&=\{(x_1, x_2)\in \mathbb{R}^2\,:\, f_1(x_1, x_2)\leqq 0, g(x_1, x_2)\leqq 0\}=\{\bar x\},
		\\
		M^2&=\{(x_1, x_2)\in \mathbb{R}^2\,:\, f_2(x_1, x_2)\leqq 0, g(x_1, x_2)\leqq 0\}=Q_0,
		\end{align*}
		we have $ T(M^1; \bar x)=\{\bar x\}$ and  $T(M^2; \bar x)=Q_0$. Thus, $T(Q_0; \bar x)=Q_0$ and
		$$\bigcap_{i=1}^2T(M^i; \bar x)=\{\bar x\}.$$
		Consequently,
		\begin{equation*}
		L(Q; \bar x) \nsubseteq \bigcap_{i=1}^2T(M^i; \bar x),
		\end{equation*}
		as required. This shows that the result of Zhao \cite[Theorem 4.1]{Zhao15}  cannot be applied for this example.
		
		Next we check the first-order necessary optimality conditions of our Corollary \ref{first_order_nec_cond}. Since $T(Q_0; \bar x)=Q_0$, we have
		$$L(Q; \bar x)\subset T(Q_0; \bar x).$$
		This means that the $(WARC)$ holds at $\bar x$. By Corollary \ref{first_order_nec_cond},  the system \eqref{first1}--\eqref{first3}
		has no solution $u\in \mathbb{R}^2$.
	}
\end{example}
\section{Concluding remarks}
\label{conclusions_sect}
In this paper we obtain primal second-order $KKT$   necessary conditions for vector optimization problems with inequality constraints in a nonsmooth setting using second-order upper generalized directional derivatives. We suppose that the objective functions and active constraints are only locally Lipschitz.  Some second-order constraint qualifications of Zangwill type, Abadie type and Mangasarian-Fromovitz type as well as a  regularity condition of Abadie type are proposed. They are applied in the optimality conditions. Our results improve and generalize the corresponding results of Aghezza et al. \cite[Theorem 3.3]{Aghezzaf99}, Gupta et al. \cite[Theorems 3.1 and 3.3]{Gupta2017}, Huy et al. \cite[Theorem 3.2]{Huy163},  Ivanov \cite[Theorem 4.1]{Ivanov15}, Constantin \cite[Theorem 2]{Elena}, Luu \cite[Corollary 5.2]{Luu17}  Zhao \cite[Theorem 4.1]{Zhao15}, and  Burachik et al. \cite[Theorem 4.3]{Rizvi12}.

To obtain  second-order $KKT$   necessary conditions in dual form, we need assume that the objective functions and constraint functions are of class $C^1(X)$. Then one can follow the scheme of the proof of \cite[Theorem 3.4]{Aghezzaf99} and we leave the details to the reader.
		
\section*{Acknowledgments}
The authors would like to thank the anonymous referee for his valuable remarks and detailed suggestions that allowed us to improve the original version. Y. B. Xiao is supported by the National Natural Science Foundation of China (11771067) and the Applied Basic Project of Sichuan Province (2019YJ0204). N. V. Tuyen is supported by  Vietnam National Foundation for Science and Technology Development (NAFOSTED) under grant number 101.01-2018.306 as well as the grant from School of Mathematical Sciences, University of Electronic Science and Technology of China, Chengdu, P.R. China.  C. F. Wen and J. C. Yao are supported by the Taiwan MOST [grant number 106-2115-M-037-001], [grant number 106-2923-E-039-001-MY3], respectively, as well as the grant from Research Center for Nonlinear Analysis and Optimization, Kaohsiung Medical University, Taiwan.	


\end{document}